\documentclass[12pt,a4paper]{amsart}
\usepackage{latexsym,amsfonts,amsthm,amsmath,mathrsfs,amssymb}
\usepackage[british]{babel}
\usepackage[dvips]{color}
\usepackage[utf8]{inputenc}
\usepackage[T1]{fontenc}
\usepackage[dvips,final]{graphics}
\usepackage{xcolor}
\usepackage{mathrsfs}
\usepackage{mathtools} 
\usepackage{breqn}
\usepackage{epstopdf}
\usepackage{verbatim}
\usepackage{amscd}
\usepackage{hyperref}
\usepackage[english,capitalize]{cleveref}
\usepackage{aliascnt}

\usepackage{todonotes}
\makeatletter
\def\newaliasedtheorem#1[#2]#3{
	\newaliascnt{#1@alt}{#2}
	\newtheorem{#1}[#1@alt]{#3}
	\expandafter\newcommand\csname #1@altname\endcsname{#3}
}
\makeatother

\theoremstyle{plain}

\newtheorem{theorem}{Theorem}[section]
\newaliasedtheorem{prop}[theorem]{Proposition}
\newaliasedtheorem{lem}[theorem]{Lemma}
\newaliasedtheorem{coroll}[theorem]{Corollary}

\theoremstyle{definition}
\newaliasedtheorem{defi}[theorem]{Definition}
\newaliasedtheorem{quest}[theorem]{Question}
\newaliasedtheorem{fact}[theorem]{Fact}

\theoremstyle{remark}
\newaliasedtheorem{rem}[theorem]{Remark}
\newaliasedtheorem{exa}[theorem]{Example}

\newcommand{\R}{\mathbb{R}}

\newcommand{\C}{\mathbb{C}}

\let\altphi\phi
\let\phi\varphi
\let\varphi\altphi
\let\altphi\undefined

\newcommand{\average}{{\mathchoice {\kern1ex\vcenter{\hrule height.4pt
width 6pt
depth0pt} \kern-9.7pt} {\kern1ex\vcenter{\hrule height.4pt width 4.3pt
depth0pt}
\kern-7pt} {} {} }}

\address{\textsc{Daniela Di Donato}: 
Dipartimento di Ingegneria Industriale e Scienze Matematiche, Via Brecce Bianche,12 60131 Ancona , Universit\'a Politecnica delle Marche.}
\email{d.didonato@staff.univpm.it}

\title{Intrinsically H\"older sections in metric spaces}

\date{\today}

\author{ Daniela Di Donato }

\begin{document}

\begin{abstract}
		We introduce a  notion of  intrinsically H\"older graphs in metric spaces. Following a recent paper of Le Donne and the author, we prove some relevant results as the Ascoli-Arzel\`a compactness Theorem, Ahlfors-David regularity and the Extension Theorem for this class of sections. In the first part of this note, thanks to Cheeger theory, we define suitable sets in order to obtain a vector space over $\R$ or $\C,$ a convex set and an equivalence relation for intrinsically H\"older graphs. These last three properties are new also in the Lipschitz case. Throughout the paper, we use basic mathematical tools.
	\end{abstract}
	\maketitle 
	\tableofcontents

\section{Introduction}
Starting to the seminal papers by Franchi, Serapioni and Serra Cassano \cite{FSSC, FSSC03, MR2032504} (see also  \cite{SC16, FS16}), Le Donne and the author  generalize the notion of intrinsically Lipschitz maps introduced in subRiemannian Carnot groups \cite{ABB, BLU, CDPT}. This concept was introduced in order to give a good definition of rectifiability in subRiemannian geometry after the negative result of Ambrosio and Kirchheim shown in \cite{AmbrosioKirchheimRect} (see also \cite{biblioMAGN1}) regarding the classical definition of rectifiability using Lipschitz maps given by Federer \cite{Federer}. The notion of rectifiable sets is a key one in Calculus of Variations and in Geometric Measure Theory. The reader can see \cite{MR2048183, MR2247905, DB22, AM2022.1, AM2022.2, M95, M61, M95b, NY18, DS91a, DS91b}.

In \cite{DDLD21} we prove some relevant statements in  metric spaces, like the Ahlfors-David regularity, the Ascoli-Arzel\'a Theorem, the Extension Theorem for the so-called intrinsically Lipschitz sections. Our point of view is to consider the graphs theory instead of the maps one. %This approach is really useful because we get short proofs with basic mathematical tools.  
 In a similar way, in this paper  we give a natural definition of intrinsically H\"older sections which includes Lipschitz ones and we prove the following results using basic mathematical tools and getting short proofs.
%we focus our attention on some relevant topics like Ascoli-Arzel\`a compactness Theorem, Ahlfors-David regularity and the Extension Theorem of a new definition of H\"older sections which is the natural next step after the introduction of intrinsically Lipschitz sections given by Le Donne and the author in metric spaces (see \cite{DDLD21}). More precisely, the main results of this paper are
\begin{enumerate}
\item  Theorem \ref{thm1}, i.e.,  Compactness Theorem a l\'a Ascoli-Arzel\`a for the intrinsically H\"older sections.
\item Theorem \ref{thm3}, i.e., Extension Theorem for the  intrinsically H\"older sections. 
\item Theorem \ref{thm2}, i.e.,  Ahlfors-David regularity for the intrinsically H\"older sections.
\item Proposition \ref{propLeibnitz formula for slope} states that the class of the intrinsically H\"older sections is a convex set.
\item Theorem \ref{theorem24} states that a suitable class of the intrinsically H\"older sections is a vector space over $\R$ or $\C.$
\item Theorem \ref{theorem24aprile} gives an equivalence relation for a suitable class of the intrinsically H\"older sections.
\end{enumerate}

The last three points are new results also in the context of Lipschitz sections.
	
 Our setting is the following. We have a metric space $X$, a topological space $Y$, and a 
quotient map $\pi:X\to Y$, meaning
continuous, open, and surjective.
%In some situations we might assume that $Y$ is metrizable so that $\pi$ becomes a Lipschitz quotient,  e.g. a submetry. normal
The standard example for us is when $X$ is a metric Lie group $G$ (meaning that the Lie group $G$ is equipped with a left-invariant distance that induces the manifold topology), for example a subRiemannian Carnot group, %$N\lhd G$ 
and $Y$ is the space of left cosets $G/H$, where 
$H<G$ is a  closed subgroup and $\pi:G\to G/H$ is the projection modulo $H$, $g\mapsto gH$.

\begin{defi}\label{def_ILS}We say that a map $\phi:Y\to X$ is an {\em intrinsically $(L,\alpha )$-H\"older section of $\pi$},  with $\alpha \in (0,1]$ and $L>0$, if
%Let $(X,d)$ be a metric space and let $Y$ be a topological space. We say that a map $\phi :Y \to X$ is a section of a quotient map $\pi :X \to Y$ if
\begin{equation*}
\pi \circ \phi =\mbox{id}_Y,
\end{equation*}
and
\begin{equation}\label{IntrinsicPROERTY}
d(\phi (y_1), \phi (y_2)) \leq L d(\phi (y_1), \pi ^{-1} (y_2))^\alpha + d(\phi (y_1), \pi ^{-1} (y_2)), \quad \mbox{for all } y_1, y_2 \in Y.
\end{equation}
Here $d$ denotes the distance on $X$, and, as usual, for a subset $A\subset X$ and a point $x\in X$, we have
$d(x,A):=\inf\{d(x,a):a\in A\}$.
\end{defi}

When $\alpha =1,$ a section $\phi$ is intrinsically Lipschitz in the sense of \cite{DDLD21}. In a natural way, the author introduced and studied other two classes of sections: the  intrinsically quasi-symmetric  sections \cite{D22.1} and the intrinsically quasi-isometric sections  \cite{D22.2} in metric spaces. 

 In Proposition \ref{Intrinsic H\"older  section.2}, we show that, when $Y$ is a bounded set,  
% an equivalent inequality of \eqref{IntrinsicPROERTY} which we will use later in order to prove some properties. More precisely, 
 \eqref{IntrinsicPROERTY} is equivalent to ask that 
\begin{equation}\label{IntrinsicPROERTY.2}
d(\phi (y_1), \phi (y_2)) \leq K d(\phi (y_1), \pi ^{-1} (y_2))^\alpha , \quad \mbox{for all } y_1, y_2 \in Y,
\end{equation}
 for a suitable $K \geq 1.$ This property will be useful in order to get Compactness Theorem a l\'a Ascoli-Arzel\'a.  % \textbf{For simplicity, throughout the paper will be considered \eqref{IntrinsicPROERTY.2} instead of  \eqref{IntrinsicPROERTY}.} 
  Moreover, we underline that, in the case $\alpha =1$ and $
 \pi$ is a Lipschitz quotient or submetry \cite{MR1736929, Berestovski}, the results trivialize, since in this case being intrinsically Lipschitz  is equivalent to biLipschitz embedding, see Proposition 2.4 in \cite{DDLD21}.
  
    \medskip

  The first series of our results is about the equicontinuity of intrinsically H\"older sections with uniform constant and consequently a compactness property \`a la Ascoli-Arzel\`a.

 \begin{theorem}[Equicontinuity and Compactness Theorem]
 \label{thm1}
 Let $\pi:X\to Y$ be a quotient map 
between a  metric space $X$ and  a topological space $Y$.   
%Consider    $L\geq 0 $,     a compact subset $K \subset X$, and  $y_0\in Y$.

\noindent{\rm \bf(\ref{thm1}.i)} Every intrinsically H\"older section of $\pi$   is continuous.

    Next, assume in addition  that closed balls in  $X$ are compact.

\noindent{\rm \bf(\ref{thm1}.ii)} For all
    $K' \subset Y$ compact,
  $L\geq 1 , \alpha \in (0,1)$,     $K \subset X$ compact, and  $y_0\in Y$
   the set
 \begin{equation*}\label{set2}
\{ \phi _{|_{K'}} : K' \to X \,| \, \phi :Y \to X \mbox{intrinsically $(L,\alpha )$-H\"older  section of $\pi$}, \phi (y_0) \in K \}
\end{equation*}
is 
 equibounded, equicontinuous, and closed in the uniform-convergence topology.
 
 \noindent{\rm \bf(\ref{thm1}.iii)} For all $L\geq 1, \alpha \in (0,1) $,     $K \subset X$ compact, and  $y_0\in Y$
the set
 \begin{equation*}\label{set1}
\{ \phi : Y \to X \,| \, \phi \text{ intrinsically $(L,\alpha )$-H\"older  section of $\pi$},
 \phi (y_0) \in K \}
\end{equation*}
  is compact with respect to  the topology of uniform convergence on compact sets.
  \end{theorem}

 Another crucial property of H\"older sections is that under suitable assumptions they can be extended. This property is much studied in the context of metric spaces if we consider the H\"older maps; the reader can see \cite{ALPD20, LN05, O09} and their references. Our proof follows using the link between H\"older sections  and level sets of suitable maps. This idea is widespread in the context of subRiemannian Carnot groups (see, for instance, \cite{ASCV, ADDDLD, DD19,  Vittone20}). In next result, we say that a map $f$ on $X$ is $L$-{\em biLipschitz on  fibers} (of $\pi$)  if 
 on each fiber of $\pi$ it restricts to an $L$-biLipschitz map.
 
% Our result is as follows.
\begin{theorem}[Extensions as level sets]\label{thm3}    
%Let $X,Y $ be metric spaces and let $\pi :X \to Y$ a quotient map.
 Let $\pi:X\to Y$ be a quotient map 
between a  metric space $X$ and  a topological space $Y$. 

\noindent{\rm \bf(\ref{thm3}.i)}
If $Z$ is a metric space, $z_0\in Z$ and $f:X\to Z$ is  $(\lambda , \beta)$-H\"older and $\lambda $-biLipschitz on  fibers, with $\lambda >0$ and $\beta \in (0,1)$, then
%\begin{enumerate}
%\item  every fiber of $\pi$ is Lipschitz equivalent to $\R^n;$
%\item $f$ is L-Lipschitz map with $L\geq 1;$ 
%\item $\forall y \in Y , \forall x,x' \in \pi ^{-1} (y),$ we have
%\begin{equation*}
%|f(x) - f(x')| \geq \frac 1 L d_X(x,x').
%\end{equation*}
%\end{enumerate}
  there exists an intrinsically $(\lambda ^2, \beta )$-H\"older section  $\phi : Y\to X $ of $\pi$ such that
\begin{equation}\label{equationluogoz0_intro}
\phi (Y)=f^{-1} (z_0).
\end{equation} 

\noindent{\rm \bf(\ref{thm3}.ii)}
Vice versa, assume that $X$ is geodesic and that there exist $k\geq 1, \alpha \in (0,1),$ $\rho : X \times X \to \R$ k-biLipschitz equivalent to the distance of $X,$ and $\tau :X \to \R$ is $ ( k,\alpha )$-H\"older and $k$-biLipschitz on  fibers such that  \begin{enumerate}
\item  for all $\tau_0\in \R $ the set $\tau ^{-1}(\tau_0)$ is an intrinsically $(k,\alpha )$-H\"older graph of a section $\phi_{\tau_0}:Y\to X;$
\item for all $x_0\in \tau ^{-1} (\tau_0)$ the map $X\to \R, x \mapsto \delta _{\tau_0} (x) := \rho (x_0, \phi_{\tau_0}(\pi (x)))$ is $k$-Lipschitz on the set $\{|\tau| \leq \delta _{\tau_0} \}$.
\end{enumerate}
Let $Y'\subset Y$ a set and $L\geq1$.
Then for every  intrinsically $(L,\alpha )$-H\"older section  $\phi : Y'\to \pi^{-1}(Y') $  of $\pi|_{\pi^{-1}(Y')} :\pi^{-1}(Y') \to  Y' $, there exists a  map $ f:X\to \R$    that is $(K,\alpha )$-H\"older and $K$-biLipschitz on  fibers, with $K\geq 1,$ such that     \begin{equation}\label{equationluogozeri_intro}
\phi (Y')\subseteq f^{-1} (0).
\end{equation}
In particular, each `partially defined' intrinsically H\"older graph $\phi (Y')$ is a subset of a `globally defined' intrinsically H\"older graph $ f^{-1} (0)$.
 \end{theorem}

%We remark that if in Theorem~\ref{thm3} the metric space $Z$ is a Carnot group $G$, one could weaken the assumption asking that $f$ is a biLipschitz embedding and that each fiber is biLipschitz equivalent to $G$. This is because every biLipschitz embedding of a Carnot group into itself is a biLipschitz equivalence, see Proposition~\ref{}.

%$Y$ is metrizable so that $\pi :X \to Y$ is a $k$-Lipschitz quotient, with $k\geq 1$, and $\phi : Y\to X $ is an intrinsically $L$-Lipschitz section  of $\pi$, 
%then the function 
%$$x\in X\mapsto f_0(x):=d(x, \phi (\pi(x)))\in \R$$
%is $(1+kL)$-Lipschitz and~\ref{equationluogozeri_intro} holds. If, moreover, there exists an $\tilde L$-Lipschitz map $\tilde f:X\to \R^n$    that on each fiber of $\pi$   restricts to an $ \tilde L$-biLipschitz map, with $\tilde L>1$, then the map
%$$x\in X\mapsto f(x):=\tilde f(x)-  \tilde f (\phi (\pi(x)))\in \R^n$$
%is $ L'$-Lipschitz map,  on each fiber of $\pi$ it restricts to an $ L'$-biLipschitz map, with $L'= \tilde L +k L $, and~\ref{equationluogozeri_intro} holds.
We underline that an important point is that the constant $\beta$ in (\ref{thm3}.i) does not change. 
Following again \cite{DDLD21}, we can prove an Ahlfors-David regularity for the intrinsically H\"older sections.  Recall that in Euclidean case $\R^s,$ there are $(L, \alpha)$-H\"older maps such that the $(s+1-\alpha)$-Hausdorff measure of their graphs is not zero and the $(s+1)$-Hausdorff measure of their graphs is zero, we give the following result. 
    
  \begin{theorem}[Ahlfors-David regularity]\label{thm2}
 Let $\pi :X \to Y$ be a quotient map between a metric space $X$ and a topological space $Y$ such that there is a measure $\mu$ on $Y$ such that for every $r_0>0$ and every $x,x' \in X$ with $\pi (x)=\pi(x')$  there is $C>0$ such that 
   \begin{equation}\label{Ahlfors27ott.112}
\mu (\pi (B(x,r))) \leq C \mu (\pi (B(x',r))),\quad \forall  r\in (0,r_0).
\end{equation}
   
Let $\ell \in (0,\infty).$ We also assume that there is an intrinsically  $(L, \alpha )$-H\"older section $\phi :Y \to X$  of $\pi$ such that $\phi (Y)$ is locally $(\ell +1-\alpha)$-Ahlfors-David regular with respect to  the measure $\phi_* \mu$. 
%  the measure $\phi_* \mu$ satisfies the following condition: for each point $x\in \phi (Y)$ there exists $c_0>0$ so that
%  \begin{equation}\label{Ahlfors_IN_N}
% \phi_* \mu \big( B(x,r) \cap \phi (Y)\big) \leq c_0 r^{\ell +1-\alpha}, \qquad \text{ for all }r\in  (0,r_0).
%\end{equation}

  Then, for every intrinsically  $(L, \alpha)$-H\"older section $\psi :Y \to X$ of $\pi,$  the set $ \psi (Y) $ is locally $\alpha(\ell +1-\alpha)$-Ahlfors-David regular with respect to  the measure $\psi_* \mu.$ 
%    the measure $\psi_* \mu$ is such that  for each point $x\in \psi (Y)$ there exist  $c_1>0$ so that
%  \begin{equation}\label{Ahlfors_IN_N}
% \psi_* \mu \big( B(x,r) \cap \psi (Y)\big) \leq c_1 r^{\alpha ({\ell +1-\alpha})}, \qquad \text{ for all }r\in  (0,r_0).
%\end{equation}

    \end{theorem} 
    
     Namely, in Theorem~\ref{thm2} locally $Q$-Ahlfors-David regularity means that    the measure $\phi_* \mu$ is such that  for each point $x\in \phi (Y)$ there exist $r_0>0$ and $C>0$ so that
  \begin{equation}\label{Ahlfors_IN_N}
 C^{-1}r^Q\leq  \phi_* \mu \big( B(x,r) \cap \phi (Y)\big) \leq C r^Q, \qquad \text{ for all }r\in (0,r_0).
\end{equation}
The same inequality will hold for   $\psi_* \mu $ with a possibly different value of $C$ and $Q$. See Section~\ref{theoremAhlforsNEW}. 

Finally, following Cheeger theory \cite{C99} (see also \cite{K04, KM16}), we give an  equivalent property of H\"older section which we will prove in  Section \ref{sec:equiv_def}. Here it is fundamental that $Y$ is a bounded set. %We underline that it  represents the first step in order to formulate Proposition \ref{propLeibnitz formula for slope}, Theorem \ref{theorem24} and Theorem \ref{theorem24aprile}:
 \begin{prop}\label{linkintrinsicocneelip}
   Let $X  $ be a metric space, $Y$ a topological and bounded space, $\pi :X \to Y$   a  quotient map, $L\geq 1$ and $\alpha , \beta , \gamma \in (0,1)$.
Assume that   every point $x\in X$ is contained in the image of an intrinsic $(L,\alpha )$-H\"older section $\psi_x$ for $\pi$.
 Then for every section $\phi :Y\to X$ of $\pi$ the following are equivalent:
   \begin{enumerate}
\item for all $x\in\phi(Y)$ the section $\phi $ is intrinsically $(L_1,\beta )$-H\"older with respect to  $\psi_x$ at   $x;$
\item  the section $\phi $  is intrinsically $(L_2,\gamma)$-H\"older.
\end{enumerate}
   \end{prop}

 {\bf Acknowledgements.}  We would like to thank Davide Vittone for helpful suggestions and Giorgio Stefani for the reference  \cite{KM16}.

\section{Equivalent definitions for intrinsically  H\"older  sections}
\label{sec:equiv_def}
 
\begin{defi}[Intrinsic H\"older section]\label{Intrinsic H\"older  section}
Let $(X,d)$ be a metric space and let $Y$ be a topological space. We say that a map $\phi :Y \to X$ is a {\em section} of a quotient map $\pi :X \to Y$ if
\begin{equation*}
\pi \circ \phi =\mbox{id}_Y.
\end{equation*}
Moreover, we say that $\phi$ is an {\em intrinsically $(L, \alpha)$-H\"older  section} with constant $L>0$ and $\alpha \in (0,1)$ if in addition
\begin{equation}\label{Intrinsic H\"older per semplificare}
d(\phi (y_1), \phi (y_2)) \leq L d(\phi (y_1), \pi ^{-1} (y_2))^\alpha +  d(\phi (y_1), \pi ^{-1} (y_2)), \quad \mbox{for all } y_1, y_2 \in Y.
\end{equation}

Equivalently, we are requesting that
\begin{equation*}
d(x_1, x_2) \leq L d(x_1, \pi ^{-1} (\pi (x_2)))^\alpha + d(x_1, \pi ^{-1} (\pi (x_2))), \quad \mbox{for all } x_1,x_2 \in \phi (Y) .
\end{equation*}
\end{defi}

We further rephrase the definition as saying that $\phi(Y)$, which we call the {\em graph} of $\phi$, avoids some particular sets (which depend on $\alpha , L$ and $\phi$ itself):

%Given $x\in X$ and $L\geq 0$ we define the following two sets:
%\begin{eqnarray*}
%R_{x,L} & = & \{ x'\in X \;|\;   d(x', \pi ^{-1} (\pi (x))) < L d(x', x)\}, \\ 
%R'_{x,L} & = & \{ x'\in X \;|\;   d(\pi ^{-1} (\pi (x')), \pi ^{-1} (\pi (x))) < L d(x', x)\}.\\
%R''_{x,L} & = & \{ x'\in X \;|\;   d(\pi (x'), \pi (x)) < L d(x', x)\},\\
%\end{eqnarray*}
%where the last one makes sense only if $Y$ is metrized.

\begin{prop}\label{propo_ovvia} Let $\pi :X \to Y$  be a  quotient map between a metric space and a topological space, $\phi: Y\to X$ be a section of $\pi$, $\alpha \in (0,1)$ and $L>0$.
Then $\phi$ is intrinsically $(L, \alpha)$-H\"older if and only if
\begin{equation*}
 \phi  (Y) \cap  R_{x,L} = \emptyset , \quad \mbox{for all } x \in \phi (Y),
\end{equation*}
where $$R_{x,L} := \left\{ x'\in X \;|\;   L d(x', \pi ^{-1} (\pi (x)))^\alpha + d(x', \pi ^{-1} (\pi (x)))  <   d(x', x)\right\}.$$

%(ii) Assuming that $Y$ is metrizable so that $\pi$ is a $k$-Lipschitz quotient, then 

%(ii) If $\phi$ is $L$-intrinsically Lipschitz, then %since $R_{x,L} \subset R'_{x,L}, x' \ne R'_{x,L} then $Ld(x,x') \leq d(\pi ^{-1} (\pi (x')), \pi ^{-1} (\pi (x))) \leq d(x', \pi ^{-1} (\pi (x)))$
%\begin{equation*}\label{coni_LipQuo1} 
% \phi  (Y) \cap  R'_{x,L} = \emptyset , \quad \mbox{for all } x \in \phi (Y).
%\end{equation*}
%and
%\begin{equation}\label{coni_LipQuo1} Daniela scrivere con R''
%\end{equation}
%
%(ii$_1$) Vice versa, if either \eqref{coni_LipQuo1}  or \eqref{coni_LipQuo2}  hold, then $\phi$ is $kL$-intrinsically Lipschitz.
\end{prop}

Proposition \ref{propo_ovvia} is a triviality, still its purpose is to stress the analogy with the intrinsically Lipschitz sections theory introduced in \cite{DDLD21} when $\alpha =1$.  In particular, the sets $R_{x,L}$ are the intrinsic cones in the sense of  Franchi, Serapioni and Serra Cassano when $X$ is a  subRiemannian Carnot group and $\alpha =1.$ The reader can see \cite{D22.3} for a good notion of intrinsic cones in metric groups.
%Indeed, the sets $R_{x,L}$ are the intrinsic cones considered in Carnot groups, see Section \ref{sec:various_notions}.
%{\bf METTER QUI CHE SE PI LIP QUOT ALLORA LIP EMBED}

Definition \ref{Intrinsic H\"older  section} it is very natural if we think that what we are studying graphs of appropriate maps. However, in the following proposition, we introduce an equivalent condition of \eqref{Intrinsic H\"older per semplificare} when $Y$ is a bounded set.

\begin{prop}\label{Intrinsic H\"older  section.2}
Let $\pi: X \to Y$ be a quotient map between a metric space $X$ and a topological and bounded space $Y$ and let $\alpha \in (0,1).$ The following are equivalent:
\begin{enumerate}
\item there is $L >0$ such that
\begin{equation*}
d(\phi (y_1), \phi (y_2)) \leq L d(\phi (y_1), \pi ^{-1} (y_2))^\alpha +  d(\phi (y_1), \pi ^{-1} (y_2)), \quad \mbox{for all } y_1, y_2 \in Y.
\end{equation*}
\item there is $K \geq 1$ such that
\begin{equation}\label{equationEQUIV}
d(\phi (y_1), \phi (y_2)) \leq K d(\phi (y_1), \pi ^{-1} (y_2))^\alpha , \quad \mbox{for all } y_1, y_2 \in Y.
\end{equation}
\end{enumerate}
\end{prop}

  \begin{proof}  

$(1) \Rightarrow (2).$ This is trivial when $d(\phi (y_1), \pi ^{-1} (y_2)) \leq 1$. On the other hand, if we consider $y_1, y_2 \in Y$ and $\bar x \in X$ such that $d(\phi (y_1), \pi ^{-1} (y_2)) = d(\phi (y_1), \bar x) >1,$ then it is possible to consider $\ell$ equidistant points $x_1,\dots , x_\ell \in X$ such that $$d(\phi (y_1), \bar x) =d(\phi (y_1), x_1)+ \sum _{i=1}^{\ell -1}d(x_i, x_{i+1} ) + d(x_\ell , \bar x),$$ with $d(\phi (y_1), x_1)= d(x_i,x_{i+1})= d(x_\ell , \bar x) \in (\frac 1 2 , 1).$ Here, $\ell \leq  \lfloor d(\phi (y_1), \bar x) \rfloor  +1$ depends on $y_1,y_2.$ Here $ \lfloor  k \rfloor  $ is the integer part of $k.$ However, it is possible to choose $k\in \R ^+$ defined as
\begin{equation}\label{costuniversale}
k:= \sup _{y_1,y_2 \in Y}  d(\phi (y_1), \pi ^{-1} (y_2)),
\end{equation}
such that $k$ not depends on the points and $\ell \leq k.$
 We notice that this constant is finite because, on the contrary, we get the contradiction $\infty= d(\phi (y_1), \pi ^{-1} (y_2))  \leq d(\phi (y_1), \phi (y_2)).$ Hence,
\begin{equation*}
\begin{aligned}
d(\phi (y_1), \phi (y_2)) & \leq L d(\phi (y_1), \bar x)^\alpha +  d(\phi (y_1), \bar x)\\
& = L d(\phi (y_1), \bar x)^\alpha + d(\phi (y_1), x_1)+ \sum _{i=1}^{\ell -1}d(x_i, x_{i+1} ) + d(x_\ell , \bar x)\\
& \leq L d(\phi (y_1), \bar x)^\alpha + d(\phi (y_1), x_1)^\alpha + \sum _{i=1}^{\ell -1}d(x_i, x_{i+1} )^\alpha  + d(x_\ell , \bar x)^\alpha \\
& \leq (L+3(\lfloor  k \rfloor +1)) d(\phi (y_1), \bar x)^\alpha \\
&=: Kd(\phi (y_1), \bar x)^\alpha .
\end{aligned}
\end{equation*}

$(2) \Rightarrow (1).$ This is a trivial implication.

\end{proof}

\begin{defi}[Intrinsic H\"older  with respect to  a section]\label{defwrtpsinew}
 Given  sections %$X  $ be a metric space, $Y$ a topological space, $\pi :X \to Y$   a  quotient and
  $\phi, \psi :Y\to X$   
  %an intrinsically Lipschitz 
  of $\pi$. We say that   $\phi $ is {\em intrinsically $(L, \alpha)$-H\"older with respect to  $\psi$ at point $\hat x$}, with $L>0, \alpha \in (0,1)$ and $\hat x\in X$, if
\begin{enumerate}
\item $\hat x\in \psi(Y)\cap \phi (Y);$
\item $\phi  (Y) \cap  C_{\hat x,L}^{\psi} = \emptyset ,$
\end{enumerate}
where
$$ C_{\hat x,L}^{\psi} := \{x\in X \,:\,  d(x, \psi (\pi (x))) > L d(\hat x, \psi (\pi (x)))^\alpha + d(\hat x, \psi (\pi (x))) \}.   $$
\end{defi}

  \begin{rem} Definition~\ref{defwrtpsinew} can be rephrased as follows.
 A section $\phi  $ is intrinsically $(L,\alpha)$-H\"older with respect to  $\psi$ at point $\hat x$ if
 and only if 
 there is $\hat y\in Y$ such that  $\hat x= \phi (\hat y)=\psi(\hat y)$ and
%  for every $x\in \phi (Y)$ %we have that $x\notin C_{x_0,L}^{\phi_0},$ i.e.,
\begin{equation}\label{defintrlipnuova}
 d(x, \psi (\pi (x))) \leq L d(\hat x, \psi (\pi ( x)))^\alpha + d(\hat x, \psi (\pi (x))), \quad \forall x \in \phi (Y), 
\end{equation}
which equivalently means %for any $y\in Y$
%On the other hand, if $x_0=\phi _0(y_0)=\phi (y_0),$ then  \eqref{defintrlipnuova} is equivalent to ask for $x=\phi (y)$
\begin{equation}\label{equation28.0}
 d(\phi (y), \psi (y)) \leq L d(\psi(\hat y), \psi (y))^\alpha +d(\psi(\hat y), \psi (y)) ,\qquad \forall y\in Y. 
\end{equation}
  \end{rem}  

Notice that Definition \ref{defwrtpsinew} does not induce an equivalence relation because of lack of symmetry in the right-hand side of \eqref{equation28.0}. In  Section \ref{An equivalence relation} we give a stronger definition in order to obtain an equivalence relation.

The proof of Proposition $\ref{linkintrinsicocneelip}$ is an immediately consequence of the following result.

 \begin{prop}\label{propIntHoldervswrt}
   Let $X  $ be a metric space, $Y$ a topological and bounded space, and $\pi :X \to Y$   a  quotient map. Let $L\geq1$ and $y_0\in Y$. Assume $\phi_0:Y\to X$ is an intrinsically $(L,\alpha )$-H\"older section of $\pi$. Let $\phi :Y\to X$ be a section of $\pi$
   such that 
   $x_0:=\phi (y_0)=\phi _0(y_0).$
   Then the following are equivalent:
   \begin{enumerate}
\item  For some $L_1\geq1$ and $\beta \in (0,1)$, $\phi $ is intrinsically $(L_1,\beta)$-H\"older with respect to  $\phi_0$ at   $x_0;$
\item  For some $L_2\geq1$ and $\gamma \in (0,1)$, $\phi $ satisfies   %is intrinsically $L_2$-Lipschitz at point $q=\phi (y)=\phi _0(y),$ i.e., 
\begin{equation}\label{equation3nov2021}
d(x_0,\phi (y)) \leq L_2d(x_0, \pi^{-1} (y))^\gamma, \quad \forall y\in Y.
\end{equation}
%where $\beta = \alpha $ if $d(x_0, \pi^{-1} (y))>1,$ otherwise $\beta =\alpha ^2.$
\end{enumerate}
   Moreover, the constants $L_1$ and $L_2 $ are quantitatively related in terms of $L$.
   \end{prop}

    \begin{proof}  We begin recall that, by Proposition \ref{Intrinsic H\"older  section.2}, \eqref{IntrinsicPROERTY} is equivalent to \eqref{equationEQUIV}.
   
$(1) \Rightarrow (2).$ For every $y\in Y,$ it follows that
     \begin{equation*}
\begin{aligned}
d(\phi (y), x_0) & \leq d(\phi (y), \phi _0(y))  + d(\phi _0(y),x_0)  \\
& \leq L_1d(\phi _0(y), x_0)^\beta +  d(\phi _0(y),x_0)  \\
& \leq L_1Ld(x_0, \pi^{-1} (y))^{\beta \alpha } + Ld(x_0, \pi^{-1} (y))^\alpha   \\
& \leq L(L_1+1) \max\{ d(x_0, \pi^{-1} (y))^{\beta \alpha }, d(x_0, \pi^{-1} (y))^\alpha  \} \\
%& \leq L(L_1+1)d(x_0, \pi^{-1} (y_1)),  \\
\end{aligned}
\end{equation*}
   where in the first inequality we used the triangle inequality, and in the second one the intrinsic H\"older property (1) of $\phi.$ Then, in the third inequality we used the intrinsic H\"older property of $\phi_0.$ Here, noticing that $\beta \alpha <\alpha ,$ we have that if $d(x_0, \pi^{-1} (y)) \leq 1,$ then $\max\{ d(x_0, \pi^{-1} (y))^{\beta \alpha }, d(x_0, \pi^{-1} (y))^\alpha  \} =d(x_0, \pi^{-1} (y))^{\beta \alpha }.$ On the other hand, if $d(x_0, \pi^{-1} (y)) > 1,$ then  using a similar technique using in Proposition \ref{Intrinsic H\"older  section.2} we obtain the same maximum with additional constant $K:= L+3(\lfloor  k \rfloor +1)$ where $k\in \R$ is given by \eqref{costuniversale}. Definitely,
       \begin{equation*}
\begin{aligned}
d(\phi (y), x_0) & \leq LK(L_1+1)  d(x_0, \pi^{-1} (y))^{\beta \alpha }.\\
%& \leq L(L_1+1)d(x_0, \pi^{-1} (y_1)),  \\
\end{aligned}
\end{equation*} 
   
      $(2) \Rightarrow (1).$   For every $y\in Y,$ we have that
\begin{equation*}
\begin{aligned}
d(\phi (y), \phi _0(y)) & \leq d(\phi (y),x_0) +  d( x_0, \phi _0 (y)) \\
%& \leq L _2 d_X(  \phi (y)  ,  \pi ^{-1} (y_1) )+ L d_X(  \phi _0(y)  ,  \pi ^{-1} (y_1) )\\
& \leq L_2  d(  \phi _0 (y)  ,  x_0)^\gamma +  d(  \phi _0 (y)  ,  x_0),
\end{aligned}
\end{equation*}
  where in the first equality we used the triangle inequality, and in the second one we used \eqref{equation3nov2021}.  (Here it is better to use the first definition of intrinsically H\"older sections).
   \end{proof}

It is easy to see that if $\alpha =1,$ then we get $\beta =\gamma$ and so we have the following corollary.
 \begin{coroll}
  Let $X  $ be a metric space, $Y$ a topological space, $\pi :X \to Y$   a  quotient map, $L\geq 1$ and $\beta \in (0,1)$.
Assume that   every point $x\in X$ is contained in the image of an intrinsic $L$-Lipschitz section $\psi_x$ for $\pi$.
 Then for every section $\phi :Y\to X$ of $\pi$ the following are equivalent:
   \begin{enumerate}
\item for all $x\in\phi(Y)$ the section $\phi $ is intrinsically $(L_1,\beta )$-H\"older with respect to  $\psi_x$ at   $x;$
\item  the section $\phi $  is intrinsically $(L_2,\beta)$-H\"older.
%   Let $X  $ be a metric space, $Y$ a topological space, and $\pi :X \to Y$   a  quotient map. Let $L\geq1$ and $y_0\in Y$. Assume $\phi_0:Y\to X$ is an intrinsically $L$-Lipschitz section of $\pi$. Let $\phi :Y\to X$ be a section of $\pi$
%   such that 
%   $x_0:=\phi (y_0)=\phi _0(y_0).$
%   Then the following are equivalent:
%   \begin{enumerate}
%\item  For some $L_1\geq1$, $\phi $ is intrinsically $(L_1,\beta)$-H\"older with respect to  $\phi_0$ at   $x_0;$
%\item  For some $L_2\geq1$, $\phi $ satisfies   %is intrinsically $L_2$-Lipschitz at point $q=\phi (y)=\phi _0(y),$ i.e., 
%\begin{equation}\label{equation3nov2021}
%d(x_0,\phi (y)) \leq L_2d(x_0, \pi^{-1} (y))^\beta, \quad \forall y\in Y.
%\end{equation}
%%where $\beta = \alpha $ if $d(x_0, \pi^{-1} (y))>1,$ otherwise $\beta =\alpha ^2.$
\end{enumerate}
% %  Moreover, the constants $L_1$ and $L_2 $ are quantitatively related in terms of $L$.
   \end{coroll}

\subsection{Continuity } An intrinsically $(L,\alpha )$-H\"older  section $\phi :Y \to X$ of $\pi$ is a continuous section. Indeed, fix a point $y\in Y$. Since $\pi$ is open, for every $\varepsilon \in (0,1)$ and every $x\in X$ such that  $x=\phi (y)$ we know that there is an open neighborhood $U_\varepsilon $ of $\pi (x)=y$ such that
\begin{equation*}
U_\varepsilon \subset \pi (B(x,[\varepsilon /(L+1)]^{1/\alpha} )).
\end{equation*}
Hence, if $y'\in U _\varepsilon $ then there is $x'\in B(x, [\varepsilon /(L+1)]^{1/\alpha} )$ such that $\pi(x')=y'$. That means $x'\in \pi^{-1} (y')$ and, consequently,
\begin{equation*}
d(\phi (y), \phi (y')) \leq L d(\phi (y), \pi ^{-1} (y'))^\alpha +d(\phi (y), \pi ^{-1} (y'))  \leq (L+1)d(x,x')^\alpha \leq \varepsilon , 
\end{equation*} 
i.e., $\phi (U_\varepsilon) \subset B(x,\varepsilon ).$

\section{Properties of linear and quotient map} In order to give some relevant properties as convexity and being vector space over $\R$ we need to ask that $\pi$ is also a linear map. Notice that this fact is not too restrictive because in our idea $\pi$ is the 'usual' projection map. More precisely, throughout the section we will consider $\pi$ a linear and quotient map between a normed space $X$ and  a topological space $Y$.  

\subsection{Basic properties}
In this section we give two simple results in the particular case when $\pi$ is a linear map. 
 \begin{prop}\label{propconvex24}  
Let $\pi: X\to Y$ be a linear and quotient map between a metric space $X$ and a topological space $Y.$ The set of all section of $\pi$ is a convex set.
\end{prop}

 \begin{proof} 
Fix $t\in [0,1]$ and let $\phi , \psi :Y \to X$ sections of $\pi.$ By the simply fact
\begin{equation*}
\begin{aligned}
\pi (t\phi (y) + (1-t) \psi (y)) = t\pi (\phi (y)) + (1-t) \pi (\psi (y)) =y,
\end{aligned}
\end{equation*}
we get the thesis.
\end{proof}

 \begin{prop}\label{propconvex30}  
Let $\pi: X\to Y$ be a linear and quotient map between a normed space $X$ and a topological space $Y.$ If $\phi: Y \to X$ is an intrinsically H\"older section of $\pi,$ then for any $\lambda \in \R-\{0\}$ the section $\lambda \phi$ is also intrinsic H\"older for $1/\lambda \pi$ with the same Lipschitz constant up to the constant $|\lambda |^{1-\alpha }$. 
\end{prop}

 \begin{proof} 
Fix $\lambda \in \R-\{0\}.$ The fact that $\lambda \phi $ is a section is trivial using the similar argument of Proposition \ref{propconvex24}. On the other hand, for any $y_1, y_2 \in Y$
  \begin{equation*}
\begin{aligned}
\| \lambda \phi (y_1) -\lambda \phi (y_2) \| \leq |\lambda | L d( \phi (y_1),  \pi^{-1} (y_2) )^\alpha  = |\lambda |^{1-\alpha } Ld(\lambda \phi (y_1),  (1/\lambda \pi)^{-1} (y_2)) ^\alpha ,
\end{aligned}
\end{equation*}
i.e., the thesis holds. This fact follows by these observations:
\begin{enumerate}
\item  if $ d( \phi (y_1),  \pi^{-1} (y_2) )=  d( \phi (y_1), a)$ then $ |\lambda |^\alpha d( \phi (y_1),  \pi^{-1} (y_2) ) ^\alpha = \|\lambda \phi (y_1)-\lambda a\| ^\alpha .$
\item $\lambda a \in \pi ^{-1} (\lambda y).$
\item $ \pi ^{-1} (\lambda y)=  (1/\lambda \pi)^{-1} ( y).$
\end{enumerate}
The second point is true because using the linearity of $\pi$ we have that $\pi (\lambda a) = \lambda \pi(a)=\lambda y.$ Finally, the third point holds because
\begin{equation*}
 \pi^{-1} (\lambda y) =\{ x \in X \,:\,  \pi (x) = \lambda y\} = \{ x \in X \,:\, 1/ \lambda \pi (x) =  y\}= (1/ \lambda \pi)^{-1}  (y),
\end{equation*}
as desired.
\end{proof}
  
   \subsection{Convex set} In this section we show that the set of all intrinsically H\"older sections is a convex set. We underline that the hypothesis on boundness of $Y$ is not necessary.
       \begin{defi}[Intrinsic H\"older set with respect to $\psi$]\label{defwrtpsinew.9apr.23} Let $\alpha \in (0,1]$ and $\psi: Y \to X$ a section of $\pi$.  We define the set  of all  intrinsically H\"older  sections  of $\pi$ with respect to  $\psi$ at point $\hat x$ as
\begin{equation*}
\begin{aligned}
H _{\psi , \hat x, \alpha} & :=\{ \phi :Y\to X \mbox{ a section of $\pi$}  \, :\, \phi \mbox{ is intrinsically $(\tilde L, \alpha )$-H\"older w.r.t. $\psi$ at point $\hat x$} \\
& \quad \quad \mbox{  for some $\tilde L>0$} \}.
 \end{aligned}
\end{equation*}
  \end{defi}
  
%An interesting observation is that, considering $H_{\psi , \hat x , \alpha },$ the intrinsic Lipschitz constant $L$ can be change but it is fundamental that the point $\hat x$ is a common one for the every sections. %Moreover, thanks to Corollary 2.7 in  \cite{D22.1}, $\alpha$ is the same for any element of $H _{\psi , \hat x}.$ Following Theorem \ref{theorem} we get the following result.
%
%   \begin{theorem}\label{theorem} Let $\pi :X \to Y$ is a linear and quotient map from a normed space $X$ to a metric space $Y.$ Assume also that $\psi :Y \to X$ is a section of $\pi$  and $\{\lambda \hat x\, :\, \lambda \in \R\} \subset X$ with $\hat x \in \psi (Y).$
%   
%Then, for any $\alpha \in (0,1),$ the set $\bigcup _{\lambda \in \R  -\{0\}} H _{\lambda \psi , \lambda \hat x , \alpha}$ is a vector space over $\R $ or $\C .$
%       \end{theorem}
%      

    \begin{prop}\label{propLeibnitz formula for slope} 
Let $\pi :X \to Y$ be a linear and quotient map with $X$ a normed space and $Y$ a topological space. Assume also that $\alpha \in (0,1],$ $\psi: Y \to X$ a section of $\pi$ and $\hat x \in \psi (Y).$ Then, the set $H _{\psi , \hat x, \alpha} $ is a convex set.
\end{prop}

 \begin{proof} Let $\phi , \eta \in H _{\psi , \hat x, \alpha} $ and let $t\in [0,1].$ We want to show that $$w := t\phi + (1-t)\eta \in H _{\psi , \hat x, \alpha}.$$ 
Notice that, by Proposition \ref{propconvex24},  $w$ is a section of $\pi$ and $w (\bar y)= \phi (\bar y)=\eta (\bar y)=\hat x$ for some $\bar y \in Y.$ On the other hand, for every $y \in Y$ we have
\begin{equation*}
\begin{aligned}
\|w (y)-\psi ( y)\|  & =\| t(\phi (y)- \psi (y)) + (1-t) (\eta (y)-\psi(y))\|,
\end{aligned}
\end{equation*}
and so
\begin{equation*}
\begin{aligned}
\|w (y)-\psi ( y)\| & \leq t \|\phi (y)-\psi ( y)\|+ (1-t) \|\eta (y)-\psi( y)\|.
\end{aligned}
\end{equation*}
Hence, 
 \begin{equation*}
\begin{aligned}
d(w (y),\psi  (y)) & \leq tL_\phi  d(\psi(\bar y), \psi (y))^\alpha + (1-t)L_\psi  d(\psi(\bar y), \psi (y))^\alpha + d(\psi(\bar y), \psi (y))\\ & = [t(L_\phi -L_\psi) + L_\psi ]  d(\psi(\bar y), \psi (y))^\alpha +d(\psi(\bar y), \psi (y)), \\
\end{aligned}
\end{equation*}
% \begin{equation*}
%\begin{aligned}
%\frac{d(w (y),w (\bar y))}{d(w (\bar y), \pi ^{-1} (y)) ^\alpha}& \leq  t \frac{   d(\phi (y), \phi (\bar y))}{d(w (\bar y), \pi ^{-1} (y)) ^\alpha }  + 
%(1-t)\frac{ d(\eta (y),\eta (\bar y))}{d(w (\bar y), \pi ^{-1} (y)) ^\alpha }   \\
%\end{aligned}
%\end{equation*}
%where we omit the term $(f(y)-f(\bar y))(1-f( y)) d(\psi (y),\psi (\bar y)) / D$ because when we will take to the limit these terms are zero using the facts that $f\in ILS_{loc}(Z,\R)$ and $\phi (\bar y) = \psi (\bar y).$
for every $y\in Y,$ as desired.
\end{proof}

\subsection{Vector space} 
In this section we show that a 'large' class of intrinsically H\"older sections is a vector space over $\R$ or $\C.$   Notice that  it is no possible to obtain the statement for $H_{\psi , \hat x, \alpha }$ since the simply observation that if $\psi (\hat y) = \hat x$ then $\psi (\hat y ) + \psi (\hat y )  \ne \hat x.$ 
   \begin{theorem}\label{theorem24} Let $\pi :X \to Y$ is a linear and quotient map from a normed space $X$ to a topological space $Y.$ Assume also that $\psi :Y \to X$ is a section of $\pi$  and $\{\lambda \hat x\, :\, \lambda \in \R^+\} \subset X$ with $\hat x \in \psi (Y).$
   
Then, for any $\alpha \in (0,1],$ the set $\bigcup _{\lambda \in \R^+} H _{\lambda \psi , \lambda \hat x , \alpha} \cup \{ 0\}$ is a vector space over $\R $ or $\C .$
       \end{theorem}

\begin{proof} 
Let $\phi , \eta \in \bigcup_{\lambda \in \R ^+} H_{\lambda \psi , \lambda \hat x , \alpha}$ and $\beta  \in \R  -\{0\}.$ We want to show that  
\begin{enumerate}
\item $w= \phi + \eta \in \bigcup_{\lambda \in \R ^+} H_{\lambda \psi , \lambda \hat x, \alpha }.$
\item $\beta \phi \in \bigcup_{\lambda \in \R ^+} H_{\lambda \psi , \lambda \hat x, \alpha }.$
\end{enumerate}

(1). If $\phi  \in H_{\delta_1 \psi , \delta _1 \hat x , \alpha }$ and $ \eta \in  H _{\delta _2 \psi , \delta _2 \hat x, \alpha }$ for some $\delta _1, \delta _2 \in \R^+$ it holds $$w \in  H_{(\delta _1 + \delta _2) \psi , ( \delta _1 + \delta _2) \hat x}.$$ 
For simplicity, we choose $\phi , \eta \in  H_{\psi , \hat x , \alpha }$ and so it remains to prove $$w \in  H_{2 \psi , 2 \hat x , \alpha }.$$

By linear property of $\pi$, $w$ is a section of $1/2 \pi .$ On the other hand, if $\psi (\bar y) = \hat x,$ then  $w(\bar y)= \phi (\bar y)+ \eta(\bar y) = 2 \psi (\bar y) \in X.$ Moreover, using \eqref{equation28.0}, we deduce 
\begin{equation*}
\begin{aligned}
\|w(y)- 2\psi (y)\| & = \| \phi ( y)+ \eta( y) - 2 \psi (y)\|\\
& \leq \|\phi ( y)- \psi (y)\| +  \|\eta ( y)- \psi (y)\|\\
& \leq 2\max\{L_\phi , L_\eta\} \|\psi (\bar y)- \psi (y)\|^\alpha  +2 \|\psi (\bar y)- \psi (y)\| \\
& = 2 ^{1-\alpha } \max\{L_\phi , L_\eta\} \| 2 \psi (\bar y) -2 \psi ( y)\| ^\alpha + \| 2\psi (\bar y)-2 \psi (y)\|,
\end{aligned}
\end{equation*}
for any $y \in Y,$ as desired.

(2). In a similar way, it is possible to show the second point.
%$$w:=\gamma \phi +\beta \eta \in \bigcup_{\lambda \in \R -\{0\}} H_{\lambda \psi , \lambda \hat x, \alpha },$$ and, in particular, if $\phi  \in H_{\delta_1 \psi , \delta _1 \hat x , \alpha }$ and $ \eta \in  H _{\delta _2 \psi , \delta _2 \hat x, \alpha }$ it holds $$w \in  ILS _{(\gamma \delta _1 +\beta \delta _2) \psi , (\gamma \delta _1 +\beta \delta _2) \hat x}.$$ 
%For simplicity, we choose $\phi , \eta \in  H_{\psi , \hat x , \alpha }$ and so it remains to prove $$w \in  H_{(\gamma +\beta ) \psi , (\gamma +\beta ) \hat x , \alpha }.$$
%By linear property of $\pi$, $w$ is a section of $1/(\alpha +\beta )\pi .$ On the other hand, if $\psi (\bar y) = \hat x,$ then  $w(\bar y)= \gamma \phi (\bar y)+\beta \eta(\bar y) = (\gamma + \beta) \psi (\bar y) \in X.$ Moreover, using \eqref{equation28.0}, we deduce 
%\begin{equation*}
%\begin{aligned}
%\|w(y)- (\gamma + \beta )\psi (y)\| & = \|\gamma \phi ( y)+\beta \eta( y) - (\gamma +\beta ) \psi (y)\|\\
%& \leq |\gamma |\|\phi ( y)- \psi (y)\| + |\beta | \|\eta ( y)- \psi (y)\|\\
%& \leq L(|\gamma |+ |\beta |) \|\psi (\bar y)- \psi (y)\|^\alpha \\
%& = L \frac{|\gamma |+ |\beta |}{|\gamma + \beta |} \|(\gamma + \beta ) \psi (\bar y) -(\gamma + \beta ) \psi ( y)\| ^\alpha,
%\end{aligned}
%\end{equation*}
%for any $y \in Y,$ as desired.
\end{proof}

   \begin{rem} Theorem \ref{theorem24} holds also if we consider $\lambda \in \R^-$ or  $\lambda \in \R$ instead of $\R^+.$
       \end{rem}
       
   %     {\color{red}{ricontrollare  }}
       
\section{An equivalence relation}\label{An equivalence relation} In this section $X$ is a metric space, $Y$ a topological space and $\pi:X \to Y$ a quotient map (we do $not$ ask that $\pi$ is a linear map). 
 We stress that Definition~\ref{defwrtpsinew} does not induce  an equivalence relation, because of lack of symmetry in the right-hand side of \eqref{equation28.0}. As a consequence we must ask a stronger condition  in order to obtain  an equivalence relation.
\begin{defi}[Intrinsic H\"older  with respect to  a section in strong sense]\label{defwrtpsinew.1}
 Given  sections   $\phi, \psi :Y\to X$   of $\pi$. We say that   $\phi $ is {\em intrinsically $(L,\alpha)$-H\"older with respect to  $\psi$ at point $\hat x$ in strong sense}, with $L >0, \alpha \in (0,1]$ and $\hat x\in X$, if
\begin{enumerate}
\item $\hat x\in \psi(Y)\cap \phi (Y);$
\item it holds
\begin{equation}\label{equation11aprile}
 d(\phi (y), \psi (y)) \leq \min\{ L d(\psi(\hat y), \psi (y))^\alpha + d(\psi(\hat y), \psi (y)) , L d(\psi(\hat y), \phi (y))^\alpha + d(\psi(\hat y), \phi (y))\} ,
\end{equation} for every $y\in Y.$
\end{enumerate}
\end{defi}

  Now we are able to give the main theorem.
  
   \begin{theorem}\label{theorem24aprile}
   Let $\alpha \in (0,1]$ and $\pi :X \to Y$ be a quotient map from a metric space $X$ to a topological space $Y.$ Assume also that $\psi :Y \to X$ is a  section of $\pi$ and $\hat x \in X.$ Then, being intrinsically H\"older with respect to $\psi$ at point $\hat x$ in strong sense induces an equivalence relation. We will write the class of equivalence of $\psi$ at point $\hat x$ as
    \begin{equation*}
\begin{aligned}
[H_{\psi , \hat x, \alpha}]& :=\{\phi :Y\to X \, \mbox{a section of $\pi$}  :\, \phi \mbox{ is intrinsically $(\tilde L, \alpha)$-H\"older with respect to $\psi$ at}\\ 
& \qquad \mbox{point $\hat x$ in strong sense, for some $\tilde L>0$} \}.
\end{aligned}
\end{equation*}
   
      \end{theorem}

An interesting observation is that, considering $H_{\psi , \hat x, \alpha },$ the intrinsic constants $L$ can be change but it is fundamental that the point $\hat x$ is a common one for the every section. %On the other hand, the intrinsic constant $\alpha$ is the same for any element belongs to $[H_{\psi , \hat x}]$ thanks to \cite[Corollary 2.7]{D22.1}.

   \begin{proof}
We need to show:
\begin{enumerate}
\item reflexive property;
\item symmetric property;
\item transitive property;
\end{enumerate}

   (1). It is trivial that $\phi \backsim \phi.$
   
   (2). If $\phi \backsim \psi,$ then $\psi \backsim \phi .$ This follows from \eqref{equation11aprile}.
   
   (3). We know that $\phi \backsim \psi$ and $\psi \backsim \eta . $ Hence, $\hat x= \phi (\hat y)= \psi (\hat y)= \eta (\hat y).$ Moreover,  by \eqref{equation11aprile}, it holds
   \begin{equation*}\label{equation11aprile.0}
   \begin{aligned}
 d(\phi (y), \psi (y)) & \leq  \min\{L_1  d(\psi(\hat y), \psi (y))^\alpha + d(\psi(\hat y), \psi (y)), L_1 d(\psi(\hat y), \phi (y))^\alpha +d(\psi(\hat y), \phi (y))\},\\
  d(\psi (y), \eta (y)) & \leq \min\{   L_2 d(\eta(\hat y), \eta (y))^\alpha +  d(\eta(\hat y), \eta (y)),  L_2 d(\eta(\hat y), \psi (y))^\alpha + d(\eta(\hat y), \psi (y))\},\\
  \end{aligned}
\end{equation*}
for any $ y\in Y$ and, consequently, if $\tilde L =2 \max\{ L_1, L_2\},$ then
   \begin{equation}\label{equation11aprile.0}
   \begin{aligned}
 d(\phi (y), \psi (y)) & \leq  d(\phi (y), \psi (y)) +  d(\psi (y), \eta (y)) \\
  & \leq  \min\{  \tilde L d(\eta(\hat y), \eta (y))^\alpha + d(\eta(\hat y), \eta (y)),  \tilde L d(\psi(\hat y), \phi (y))^\alpha + d(\psi(\hat y), \phi (y))\},\\
   & =  \min\{ \tilde L d(\eta(\hat y), \eta (y))^\alpha + d(\eta(\hat y), \eta (y)), \tilde L d(\eta(\hat y), \phi (y))^\alpha +  d(\eta(\hat y), \phi (y))\},\\
  \end{aligned}
\end{equation}
for any $ y\in Y.$ This means that $\phi \backsim \eta ,$ as desired.
      \end{proof}

\section{An Ascoli-Arzel\`a compactness theorem}
In this section, we talk about compact subset of a metric space and so we can use Proposition \ref{Intrinsic H\"older  section.2} (and, in particular, \eqref{equationEQUIV}).
 \begin{theorem}[Compactness Theorem]
 Let $\pi:X\to Y$ be a quotient map  between a  metric space $X$ for which closed balls are compact and  a topological space $Y$.     Then,
  
\noindent{\rm \bf(i)} For all
    $K' \subset Y$ compact,
  $L\geq 1 $,   $\alpha \in (0,1), K \subset X$ compact, and  $y_0\in Y$
   the set
 \begin{equation*}\label{set2'}
\mathcal{A}_0:= \{ \phi _{|_{K'}} : K' \to X \,| \, \phi :Y \to X \mbox{ intrinsically $(L,\alpha )$-H\"older  section of $\pi$}, \phi (y_0) \in K \}
\end{equation*}
is 
 equibounded, equicontinuous, and closed in the uniform convergence topology.
 
 \noindent{\rm \bf(ii)} For all $L\geq 1 $,   $\alpha \in (0,1),   K \subset X$ compact, and  $y_0\in Y$
the set
 \begin{equation*}\label{set1'}
\{ \phi : Y \to X \,: \, \phi \text{ intrinsically $(L,\alpha )$-H\"older  section of $\pi$},
 \phi (y_0) \in K \}
\end{equation*}
  is compact with respect to  the uniform convergence on compact sets.
%   the set
% \begin{equation*}\label{set1}
%\{ \phi : U \to X \,: \, \phi \mbox{ is an intrinsically $L$-Lipschitz  section of $\pi$}\} \cap \{ \phi : U \to X \,: \,  \phi (y_0) \in K \}
%\end{equation*}
%  is compact with respect to   the uniform convergence on compact sets.
  \end{theorem}

We need the following remark proved in \cite[Remark 2.1]{DDLD21}.
 \begin{rem}\label{remnewok}
Let $\pi :X \to Y$ be an open map, $K\subset X$ be a compact set and $y\in Y.$ Then $\pi$ is uniformly open on $K \cap \pi ^{-1} (y),$ in the sense that, for every $\varepsilon >0$ there is a neighborhood $U_\varepsilon $ of $y$ such that
\begin{equation*}
U_\varepsilon \subset \pi (B(x,\varepsilon )), \quad \forall x \in K \cap \pi ^{-1} (y).
\end{equation*}
%Indeed, since $\pi $ is open, for every $x\in \pi ^{-1} (y)$ there is a neighborhood $U_{\varepsilon , x} $ of $y$ that is contained in $\pi (B( x, \frac  \varepsilon 2) ).$ Moreover, because $K$ is compact, we know that there is a finite $\frac \varepsilon 2$-net $N\subset K \cap \pi ^{-1} (y).$ Finally, if we put $U_\varepsilon := \bigcap _{ x \in N} U_{\varepsilon , x},$ we have that for all $x\in K \cap \pi ^{-1} (y)$ there is a point $\bar x \in N$ such that $d(x,\bar x) < \frac \varepsilon 2$ and
%\begin{equation}
%U_\varepsilon \subseteq  U_{\varepsilon , \bar x} \subseteq \pi (B( \bar x,  \varepsilon /2) ) \subseteq \pi (B( x,  \varepsilon ) ),
%\end{equation}
%as wished. 
  \end{rem}
  
\begin{proof}
\noindent{\rm \bf(i).} We shall prove that for all
    $K' \subset Y$ compact,
  $L\geq 1 $,  $\alpha \in (0,1),   K \subset X$ compact, and  $y_0\in Y$
   the set $\mathcal{A}_0$
% \begin{equation*}\label{set2}
%\{ \phi _{|_{K'}} : K' \to X \,: \, \phi :Y \to X \mbox{ is an intrinsically $L$-Lipschitz  section of $\pi$ and } \phi (y_0) \in K \}
%\end{equation*}
 is 
\begin{description}
\item[(a)]  equibounded;
\item[(b)] equicontinuous;
\item[(c)] closed.
\end{description}

(a). Fix a compact set $K' \subset Y$ such that  $y_0\in K'.$ We shall prove that for any $y\in K'$ 
 \begin{equation*}\label{set3}
\mathcal{A}:= \{ \phi (y) \,: \, \phi _{|_{K'}} : K' \to X \mbox{ is an intrinsically $(L,\alpha )$-H\"older  section of $\pi$ and }  \phi (y_0) \in K \}
\end{equation*}
 is relatively compact in $X.$ Fix a point $x_0 \in K$ and let $k:=$ diam$_d(K)$ which is finite because $K$ is compact in $X.$ Then, for every $\phi$ belongs to $\mathcal{A}$, we have that
 \begin{equation*}
\begin{aligned}
d(x_0,\phi (y) ) & \leq d(x_0,\phi (y_0) ) + d(\phi (y_0),\phi (y) )\\
&  \leq k + Ld(\pi ^{-1}(y), \phi (y_0) )^\alpha  \\
& \leq k +L \max _{ x\in K} d(\pi ^{-1}(y), x )^\alpha,
\end{aligned}
\end{equation*}
where in the first equality we used the triangle inequality, and in the second one we used the fact that $\phi \in \mathcal{A}$ and $x_0\in K$. Finally, in the last inequality we used again $\phi (y_0)\in K$ and that the map $X\ni x \mapsto d(\pi ^{-1}(y) ,x)^\alpha $ is a continuous map and so admits maximum on compact sets. Since closed balls on $X$  are compact, we infer  that the set $\mathcal{A}$  is relatively compact in $X,$ as desired.

(b). We shall to prove that for every $y\in K'$ and every $\varepsilon >0$ there is an open neighborhood $U_y \subset K' \subset Y$ such that for any $\phi \in \mathcal{A}$  and any $y'\in U_y,$ it follows 
\begin{equation}\label{puntob}
d(\phi (y), \phi( y')) \leq \varepsilon .
\end{equation}

Because of equiboundedness, we have that for any $\phi \in \mathcal{A}, y \in Y$ the set $\phi (y)$ lies within a compact set $K_y$ and so, by Remark~\ref{remnewok}, $\pi$ is  uniformly open on $K_y \cap \pi ^{-1} (y).$ Now let $U_\varepsilon $ an neighborhood of $y$ such that $U_\varepsilon \subset \pi (B(x,( \varepsilon /L)^{1/\alpha}))$ for every $x\in K_y \cap \pi ^{-1} (y).$ Then we want to show that such  neighborhood $U_\varepsilon $ of $y$ is the set that we are looking for. Take $y'\in U_\varepsilon $ and let $x=\phi (y).$ Hence, there is $x'\in B(x, ( \varepsilon /L)^{1/\alpha}))$ with $\pi (x')=y'$ and, consequently, $x'\in \pi ^{-1}(y').$ Thus we have that for all $\phi$ belongs to $\mathcal{A}$
 \begin{equation*}
\begin{aligned}
d(\phi (y),\phi (y') ) \leq L d(\phi (y), \pi ^{-1}(y'))^\alpha \leq Ld(x,x')^\alpha \leq L \frac \varepsilon {L}  \leq \varepsilon ,
\end{aligned}
\end{equation*} i.e., \eqref{puntob} holds. 
Finally, since the bound is independent on $\phi ,$ we proved the equicontinuity.

(c). By (a) and (b) we can apply Ascoli-Arzel\'a Theorem to the set $\mathcal{A}$. Hence, every sequence in it has a converging subsequence. Moreover, this set is closed since if $\phi _h$ is a sequence in it converging pointwise  to $\phi$, then $\phi \in \mathcal{A}.$ Indeed, taking the limit of 
\begin{equation*}
d(\phi_h (y), \phi _h (y')) \leq L d(\pi ^{-1} (y), \phi _h (y') )^\alpha ,
\end{equation*}
one gets
\begin{equation*}
d(\phi (y), \phi  (y')) \leq L d(\pi ^{-1} (y), \phi (y') )^\alpha .
\end{equation*}
Finally, it is trivial that the condition $\phi _h (y_0) \in K$ passes to the limit since $K$ is compact.

\noindent{\rm \bf(ii).} Follows from the latter point $(i)$ using Ascoli-Arzel\'a Theorem.
\end{proof}

\section{Level sets and extensions}

In this section we prove Theorem~\ref{thm3}. The proof of this statement is similar to the one of  \cite[Theorem 1.4]{DDLD21} regarding Lipschitz sections theory (see also \cite[Theorem~1.5]{Vittone20}). 
We need to mention several earlier partial results on extensions of Lipschitz graphs in the context of Carnot groups, as for example in \cite{FSSC06, DDF},   \cite[Proposition~4.8]{MR3194680}, %in the  Heisenberg group
   \cite[Proposition~3.4]{V12}, %,  for the case of codimension one in the Heisenberg group;
 \cite[Theorem~4.1]{FS16}. %  (codimensione 1 in gruppi di Carnot)
 
   Regarding extension theorems in metric spaces, the reader can see \cite{ALPD20, LN05, O09} and their references.

  \begin{proof}[Proof of Theorem~\ref{thm3} \noindent{\rm \bf(\ref{thm3}.i)}]
We begin noting that for any $y\in Y$ the map $f_{|_{\pi^{-1} (y)}} : \pi^{-1} (y) \to Z$ is biH\"older  and so for any $y\in Y$ there is a unique $x \in \pi^{-1} (y)$ such that $f(x)=z_0.$ Hence, it is natural to define $\phi (y):=x$ and so \eqref{equationluogoz0_intro} holds. Moreover, $\phi : Y\to X $ intrinsically $(\lambda^2 , \beta )$-H\"older; indeed, if we consider $x_1 \in \pi^{-1} (y_1) \cap f^{-1} (z_0)$ and $x_2 \in \pi^{-1} (y_2) \cap f^{-1} (z_0),$ then we would like to prove that
\begin{equation}\label{dis27set}
d(x_1,x_2) \leq \lambda^2 d(x_1, \pi ^{-1}( y_2))^{\beta } + d(x_1, \pi ^{-1}( y_2)).
\end{equation}
Let $\bar x_2 \in \pi ^{-1}( y_2)$ such that $d(x_1, \bar x_2)=  d(x_1, \pi ^{-1}( y_2)),$ then it follows that 
\begin{equation}\label{servenelladimVITTONE}
\begin{aligned}
d(x_1,x_2) & \leq d(x_1, \bar x_2)+  d(\bar x_2, x_2)\\
&  \leq  d(x_1,\bar x_2)+ \lambda  d(f(\bar x_2),f(x_2)) \\
& =  d(x_1,\bar x_2)+ \lambda  d(f(\bar x_2), f(x_1)) \\
& \leq d(x_1, \pi ^{-1}( y_2)) + \lambda^2 d(x_1, \pi ^{-1}( y_2))^{\beta },
\end{aligned}
\end{equation}
where in the first inequality we used the triangle inequality, and in the second inequality we used the biLipschitz property of $f$ on the fibers.  In the first equality we used the fact that $f(x_1)=f(x_2)=z_0$ and finally we used again the biH\"older property of $f$. 
Consequently, \eqref{dis27set} is true and the proof is complete.
\end{proof}

\medskip

  \begin{proof}[Proof of Theorem $\ref{thm3}$ \noindent{\rm \bf(\ref{thm3}.ii)}]  Fix $x_0 \in  \tau ^{-1} (\tau_0).$ We consider the map $f_{x_0} :X \to \R$ defined as
\begin{equation*}\label{int}
f_{x_0}(x)=\left\{ 
\begin{array}{lcl}
2(\tau (x)-\tau (x_0) - \gamma (\delta _{\tau _0} (x)^\alpha + \delta _{\tau _0} (x)) &   & \mbox{ if } \, \, |\tau (x)-\tau (x_0)| \leq 2\gamma [\delta _{\tau _0}(x)^\alpha +\delta _{\tau _0}(x) ]\\
\tau (x)-\tau (x_0) &    & \mbox{ if }\,\, \tau (x)-\tau (x_0) > 2\gamma [\delta _{\tau _0}(x)^\alpha  + \delta _{\tau _0} (x) ]\\
3( \tau (x)-\tau (x_0)) &    & \mbox{ if }\,\, \tau (x)-\tau (x_0) < -2\gamma [\delta _{\tau _0}(x)^\alpha  +\delta _{\tau _0}(x) ]\\
\end{array}
\right.
\end{equation*}
where $\gamma := 2kL+1.$%k(L+1)+1.$  
We prove that $f_{x_0}$ satisfies the following properties:
\begin{description}
\item[$(i)$]  $f_{x_0}$ is $(K,\alpha)$-H\"older;
\item[$(ii)$]  $f_{x_0}(x_0)=0;$
\item[$(iii)$]  $f_{x_0}$ is biLipschitz on fibers. %and if $\tau (x)-\tau (x')\geq 0$ then $f(x)-f(x') \geq 0.$
\end{description}
where $K= \max\{ 3k, 2k+4\gamma k \}=2k+4\gamma k$ because $\gamma >1.$ The property $(i)$ follows using that $\tau , \delta _{\tau _0}$ are H\"older and Lipschitz, respectively, and $X$ is a geodesic space. On the other hand, $(ii)$ is true noting that  $\delta _{\tau _0}(x_0)= \rho (x_0, \phi_{\tau_0}(\pi (x_0)))=0$ because $x_0 \in \phi _{\tau_0} (Y).$

Finally, for any $x,x' \in \pi ^{-1} (y) $ we have that $\rho (x_0, \phi_{\tau_0}(\pi (x)))=\rho (x_0, \phi_{\tau_0}(\pi (x')))$  and so $f$ is biLipschitz on fibers because $\tau$ is so too.  
%Moreover, if $\tau (x)-\tau (x')\geq 0$ then
%\begin{equation*}
%f(x)-f(x') =2( \tau (x)-\tau (x')) \geq 0,
%\end{equation*}
%i.e, $f$ is the nondecreasing on fibers. 
Hence $(iii)$ holds.

Now we consider the map $f:X \to \R$ given by
\begin{equation*}
f(x) := \sup _{x_0 \in \phi (Y)} f_{x_0} (x), \quad \forall x \in X,
\end{equation*}
and we want to prove that it is the map we are looking for. The H\"older property and the Lipschitz property on the fiber are true recall that the function $\delta _{x_0} $ is constant on the fibers. Consequently, the only non trivial fact to show is \eqref{equationluogozeri_intro}. 
Fix $\bar x_0 \in \phi (Y')$.
By $(ii)$ we have that $f_{ \bar x_0}( \bar x_0)=0$ and so it is sufficient  to prove that $f_{x_0}(\bar x_0)\leq 0$ for $x_0\in \phi (Y').$ Let $x_0\in \phi (Y').$  Then  using in addition that $\tau$ is $k$-Lipschitz, and that $\phi$ is  intrinsically $L$-Lipschitz, we have
\begin{equation*}
\begin{aligned}
|\tau (\bar x_0) -\tau( x_0)| & \leq k d(\bar x_0 , x_0)  \leq  kL d(x_0,  \phi_{\tau_0}(\pi (\bar x_0)) )^\alpha + kd\left(x_0,  \phi_{\tau_0}(\pi (\bar x_0)) \right)   \\
& <\gamma \left(\delta _{\tau _0} (\bar x _0)^\alpha + \delta _{\tau _0} (\bar x _0) \right),
\end{aligned}
\end{equation*}
and so
\begin{equation*}
f_{x_0}(\bar x_0) = 2(\tau (\bar x_0)-\tau (x_0) - \gamma \left(\delta _{\tau _0} (\bar x_0) ^\alpha +\delta _{\tau _0} (\bar x_0) \right) )<0,
\end{equation*}
i.e., \eqref{equationluogozeri_intro} holds.

   \end{proof}

\section{Ahlfors-David regularity}\label{theoremAhlforsNEW} 
 In this section we prove Ahlfors-David regularity for the intrinsically H\"older  sections. The proof of this statement is similar to the one of \cite[Theorem 1.3]{DDLD21} (see also \cite{D22.2} for the intrinsically quasi-isometric case). 
 
In order to give the proof, we need the following statement
  \begin{lem}\label{perHolder.303}
 Let $X$ be a metric  space, $Y$   a topological space, and $\pi:X\to Y$ a quotient map. If $\phi :Y \to X$ is an intrinsically $(L, \alpha)$-H\"older section of $\pi$ with $\alpha \in (0,1)$ and $L>0,$ then
		\begin{equation}\label{inclusionepalle.0}
\pi \left(B\left(p, r \right) \right) \subset \pi ( B(p,(L+1)r^\alpha ) \cap \phi (Y)) \subset \pi (B(p,(L+1)r^\alpha)), \quad \forall p\in \phi (Y), \forall r>0.
\end{equation}
		
  \end{lem}

  \begin{proof}
Regarding the first inclusion, fix $p\in \phi (Y), r>0$ and $q\in B(p,r).$  We need to show that $\pi (q) \in \pi (\phi (Y) \cap B(p,(L+1)r^\alpha)).$ Actually, it is enough to prove that 
\begin{equation}\label{equation2.6.1}
\phi (\pi (q)) \in B(p,(L+1)r^\alpha),
\end{equation}
 because if we take $g:= \phi (\pi (q)),$ then $g\in \phi (Y)$ and 
 \begin{equation*}
 \pi (g)= \pi (\phi (\pi (q))) =\pi (q) \in \pi (\phi (Y) \cap B(p,(L+1)r^\alpha)).
\end{equation*}
 
 Hence in a similar way to the point (i), we get that for any $p, q, g \in \phi (Y)$ with $g= \phi (\pi (q)),$
 \begin{equation*}
  \begin{aligned}
d(p,g)& \leq L d(p, \pi ^{-1} (\pi(g)))^\alpha + d(\phi (y_1), \pi ^{-1} (\pi(g))) \\
%& =  L d(p, \pi ^{-1} (\pi(q)))^\alpha + d(p, \pi ^{-1} (\pi(q))) \\
& \leq  L d(p, q)^\alpha + d(p, q) \\
& \leq ( L+1) r^\alpha ,
\end{aligned} \end{equation*}
i.e., \eqref{equation2.6.1} holds, as desired.  Finally, the second inclusion in \eqref{inclusionepalle.0} follows immediately noting that $\pi (\phi (Y))=Y$ because $\phi$ is a section and the proof is complete. %$\phi (Y)\cap B(p,r) \subset B(p,r),$ $\phi (Y)\cap B(p,r) \subset \phi (Y)$
  \end{proof}

    Now we are able to give the proof of Theorem $\ref{thm2}.$
  \begin{proof}[Proof of Theorem $\ref{thm2}$]
  Let $\phi$ and $\psi$ intrinsically $(L,\alpha)$-H\"older sections, with $L>0$ and $\alpha \in (0,1)$.
  Fix $y\in Y.$  By Ahlfors-David regularity of $\phi (y),$ we know that there is $c_1>0$ such that 
     \begin{equation}\label{AhlforsNEW0}
 \phi _* \mu \big( B(\phi (y),r) \cap \phi (Y)\big)  \leq c_1 r^{\ell +1-\alpha},
\end{equation}
for all $0\leq r \leq 1.$ We would like to show that there is $c_2>0$ such that
   \begin{equation}\label{AhlforsNEW127}
\psi _* \mu \big( B(\psi(y),r) \cap \psi (Y)\big)  \leq c_2 r^{\alpha ({\ell +1-\alpha})},
\end{equation}
  for every $0\leq r \leq 1.$    
  We begin noticing that 
%  by symmetry and  \eqref{Ahlfors27ott.112}
%   \begin{equation}\label{Ahlfors27ott}
%C^{-1} \mu (\pi (B(\psi(y),r))) \leq \mu (\pi (B(\phi(y),r))) \leq C \mu (\pi (B(\psi(y),r))).
%%, \qquad \forall y\in Y.
%\end{equation}
%Moreover, 
\begin{equation}\label{serveperAhlfors27}
  \psi _* \mu \big( B(\psi(y),r) \cap \psi (Y)\big) =  \mu ( \psi^{-1} \big( B(\psi(y),r) \cap \psi (Y)\big) ) = \mu ( \pi \big( B(\psi(y),r) \cap \psi (Y)\big) ), 
\end{equation}
  and, consequently, %by Lemma~\ref{lemAhlfors}  and using \eqref{inclusionepalle} with $\psi$ in place of $\phi$, we deduce that
\begin{equation*}
\begin{aligned} 
\psi _* \mu \big( B(\psi(y),r) \cap \psi (Y)\big) & \leq  \mu (\pi (B(\psi (y),r))) \leq C  \mu (\pi (B(\phi (y),r)))\\
& \leq C  \mu (\pi (B(\phi (y), (L+1)r^\alpha ) \cap \phi (Y)))\\
& = C  \phi _* \mu \big( B(\phi(y),  (L+1)r^\alpha ) \cap \phi (Y)\big)  \leq  {c_1} C  (L+1)^{\alpha ({\ell +1-\alpha})} r^{\alpha ({\ell +1-\alpha})}.
\end{aligned}
\end{equation*}
where in the first inequality we used  the second inclusion of \eqref{inclusionepalle.0}  with $\psi$ in place of $\phi$, and in the second one we used \eqref{Ahlfors27ott.112}. In the  third  inequality we used the first inclusion of \eqref{inclusionepalle.0} and in the fourth one we used  \eqref{serveperAhlfors27}  with $\phi$ in place of $\psi.$
%Hence, putting together the last two inequalities we have that \eqref{AhlforsNEW127} holds with\\ ${c_3} =  c_1C^{-1}  (L+1)^{-Q/\alpha} $ and $c_4 = {c_2} C  (L+1)^{\alpha Q}.$ 
  \end{proof}

 \bibliographystyle{alpha}
\bibliography{DDLD}

\end{document}